\documentclass[a4paper,11pt]{amsart}
\usepackage[centering, totalwidth = 380pt, totalheight = 600pt]{geometry}

\usepackage{amscd,amssymb,amsmath,amsfonts,amsthm, mathrsfs}
\usepackage{stmaryrd}
\usepackage{graphicx}
\usepackage{verbatim,paralist}
\usepackage{textcomp}
\usepackage[shortlabels]{enumitem}
\usepackage[colorlinks]{hyperref}
\hypersetup{%
  urlcolor=blue,linkcolor=blue,citecolor=blue
}
\usepackage{pdfpages}

\usepackage[T1]{fontenc}

\usepackage[arrow, matrix, tips, curve, graph, rotate]{xy}
\SelectTips{cm}{10}
\newcommand{\cd}[2][]{\vcenter{\hbox{\xymatrix#1{#2}}}}

\makeatletter
\def\matrixobject@{%
  \edef \next@{={\DirectionfromtheDirection@ }}%
  \expandafter \toks@ \next@ \plainxy@
  \let\xy@@ix@=\xyq@@toksix@
  \xyFN@ \OBJECT@}
\let\xy@entry@@norm=\entry@@norm
\def\entry@@norm@patched{%
  \let\object@=\matrixobject@
  \xy@entry@@norm }
\AtBeginDocument{\let\entry@@norm\entry@@norm@patched}
\makeatother

\newcommand{\twocong}[2][0.5]{\ar@{}[#2] \save ?(#1)*{\cong}\restore}
\newcommand{\twoeq}[2][0.5]{\ar@{}[#2] \save ?(#1)*{=}\restore}
\newcommand{\ltwocell}[3][0.5]{\ar@{}[#2] \ar@{=>}?(#1)+/r 0.2cm/;?(#1)+/l 0.2cm/^{#3}}
\newcommand{\ltwocello}[3][0.5]{\ar@{}[#2] \ar@{=>}?(#1)+/r 0.2cm/;?(#1)+/l 0.2cm/_{#3}}
\newcommand{\rtwocell}[3][0.5]{\ar@{}[#2] \ar@{=>}?(#1)+/l 0.2cm/;?(#1)+/r 0.2cm/^{#3}}
\newcommand{\rtwocello}[3][0.5]{\ar@{}[#2] \ar@{=>}?(#1)+/l 0.2cm/;?(#1)+/r 0.2cm/_{#3}}
\newcommand{\utwocell}[3][0.5]{\ar@{}[#2] \ar@{=>}?(#1)+/d  0.2cm/;?(#1)+/u 0.2cm/_{#3}}
\newcommand{\dtwocell}[3][0.5]{\ar@{}[#2] \ar@{=>}?(#1)+/u  0.2cm/;?(#1)+/d 0.2cm/^{#3}}
\newcommand{\ultwocell}[3][0.5]{\ar@{}[#2] \ar@{=>}?(#1)+/dr  0.2cm/;?(#1)+/ul 0.2cm/^{#3}}
\newcommand{\urtwocell}[3][0.5]{\ar@{}[#2] \ar@{=>}?(#1)+/dl  0.2cm/;?(#1)+/ur 0.2cm/^{#3}}
\newcommand{\dltwocell}[3][0.5]{\ar@{}[#2] \ar@{=>}?(#1)+/ur  0.2cm/;?(#1)+/dl 0.2cm/^{#3}}
\newcommand{\drtwocell}[3][0.5]{\ar@{}[#2] \ar@{=>}?(#1)+/ul  0.2cm/;?(#1)+/dr 0.2cm/^{#3}}

\newtheorem{theorem}{Theorem}[section]

\newtheorem{proposition}[theorem]{Proposition}

\newtheoremstyle{step}{2\bigskipamount}{\medskipamount}{\upshape}{}{\itshape}{. }{ }{\underline{Step~\thestep}}
\theoremstyle{step}

\renewcommand{\thestep}{\arabic{step}}

\theoremstyle{definition}
\newtheorem{example}[theorem]{Example}
\newtheorem{remark}[theorem]{Remark}

\newcommand{\Ra}{\Rightarrow}

\newcommand{\ldual}[1]{\mathord{{\let\nolimits\relax\sideset{^\wedge}{}{#1}}}}
\newcommand{\laction}[2]{\mathord{{\let\nolimits\relax\sideset{^{#1}}{}{#2}}}}
\newcommand{\conj}[2]{\mathord{{\let\nolimits\relax\sideset{^{#1}}{}{#2}}}}

\newcommand{\xra}{\xrightarrow}
\makeatletter
\newcommand{\xRa}[2][]{\ext@arrow 0359\Rightarrowfill@{#1}{#2}}
\makeatother

\def\CA{{\mathscr A}}
\def\CB{{\mathscr B}}
\def\CC{{\mathscr C}}
\def\CD{{\mathscr D}}

\def\CG{{\mathscr G}}

\def\CI{{\mathscr I}}

\def\CU{{\mathscr U}}
\def\CV{{\mathscr V}}

\def\CX{{\mathscr X}}

\def\dd{{\colon}}
\def\ot{{\otimes}}

\DeclareMathAlphabet{\mathbbe}{U}{bbold}{m}{n}

\newcommand{\tor}{\mathbin{\relbar\joinrel\mapstochar\joinrel\rightarrow}}

\begin{document}

\author{Richard Garner}
\address{Centre of Australian Category Theory, Macquarie University
  2109, Australia}
\email{richard.garner@mq.edu.au}
\author{Ross Street}
\address{Centre of Australian Category Theory, Macquarie University
  2109, Australia}
\email{ross.street@mq.edu.au}
\subjclass{18D20}
\keywords{weighted colimit, module, enriched category}
\title{Absolute colimits}
\date{\today}
\dedicatory{Dedicated to Bob Par\'e with gratitude for so many
  beautiful theorems}
\maketitle

\begin{abstract}
  \noindent In the context of enriched category theory, we give
  necessary and sufficient conditions for a module morphism
  $\alpha \dd M \to \CC(F,Z)$ to exhibit a functor $Z\dd \CA\to \CC$
  as an absolute $M$-weighted colimit of a functor ${F\dd \CB\to \CC}$.
  We also review, with short proofs, the various criteria for the
  weight $M$ itself to be absolute, in the sense that any $M$-weighted
  colimit is absolute. Finally, we prove that any absolute
  $M$-weighted colimit can be viewed as a colimit weighted by an
  absolute weight $M'$.
\end{abstract}

\section*{Introduction}\label{intro}

In \cite{Par1969}, Par\'e characterised conical limits of functors
$F\dd \CB\to \CC$ which are {\em absolute} in the sense that they are
preserved by all functors out of $\CC$. For example, when $\CB$ is the
free category on an idempotent, a colimit is a splitting for the
idempotent picked out by $F$, which is always absolute. But when $\CB$ is the
parallel-pair category, a colimit is a coequaliser, which is only
\emph{sometimes} absolute---for example, when it admits an enhancement
to a split coequaliser diagram. Par\'e's result says
what this ``enhancement'' should be in general.

Par\'e also characterised conical limits of functors $F\dd \CB\to \CC$
with $\CC$ additive which are {\em additively absolute}, that is,
preserved by all additive functors out of $\CC$. In this setting, it
is the
splittings of idempotents and finite direct sums
that are always absolute, while other kinds of colimit are only
still sometimes absolute---again, under circumstances delineated by
Par\'e's result.

The additive case points towards the context of enriched category
theory. In this setting, the second author~\cite{23} characterised
those weights $M$ for which every $M$-weighted colimit is absolute in
the sense that it is preserved by all (enriched) functors; the
condition is that $M$ should have a right adjoint module $M^*$.
Important examples of such enriched absolute colimit-types include
coinvariants of a finite group action (in the $\mathbb{Q}$-linear
setting), degree shift (in the dg setting), limits of Cauchy sequences
(in the setting of Lawvere's generalised metric
spaces~\cite{LawMetric}) and, in an extremely pretty and
striking result due to Walters~\cite{Walters}, glueings of
compatible families in a sheaf---with part of the prettiness being
that one must consider enrichment not just in a monoidal category, but
a \emph{bicategory}.

Yet applying the results of~\cite{23} to the situations considered
in~\cite{Par1969} does not recover Par\'e's results. In the unenriched
case, the absolute weights are generated by the splittings of
idempotents, and in the additive case, by the idempotent-splittings
and the finite direct sums: so we learn nothing about the colimit-types
like coequalisers which are sometimes, but not always, absolute.

In this article, we present a new theorem which is closer to a pure
enriched version of~\cite{Par1969}. As we will see, it is strong
enough to deduce the results in~\cite{23}, as well as those in the first
author's~\cite{Garn2014}: given a module $M$ with right adjoint $M^*$,
\emph{op.~cit.} showed the equivalence of having an $M$-weighted
colimit for a functor $F$, having an $M^*$-weighted limit for $F$, and
having an $M$-weighted cocone and an $M^\ast$-weighted cone making two
diagrams commute. (This generalises the familiar fact that direct sums
in an additive category are both coproducts and products.) 
The second author's paper \cite{GarPow} with Power is also relevant here.

To close the circle, we prove a second new theorem, showing that any
weighted colimit $F \colon \CB \rightarrow \CC$ which is absolute in
the broader sense of~\cite{Par1969} is also an absolute colimit in the
narrower sense of~\cite{23}, so long as we permit ourselves to change
the diagram and the weight. For example, a coequaliser which admits a
splitting (absolute in the sense of~\cite{Par1969}) can be recast as a
coequaliser of a contractible pair (which is absolute in the sense
of~\cite{23}).

\section{Enriched Par\'e}\label{givenfun}

We work in the context of enriched categories and weighted (sometimes called ``indexed'') colimits;
see \cite{12, KellyBook, LawMetric}. So we may capture results
like those of Walters, we work with enrichment over a \emph{bicategory}~\cite{18, 21}.
The ``presheaf'' or ``right module'' category of a category $\CC$ will be denoted by $\mathcal{P}\CC$
with Yoneda embedding functor $\mathrm{Y}\dd \CC\to \mathcal{P}\CC$. 
We also have $\mathcal{P}^{\dagger}\CD \mathbin{:=} \mathcal{P}(\CD^\mathrm{op})^\mathrm{op}$ with corresponding Yoneda embedding functor 
$\mathrm{Y}^{\dagger}\dd \CD\to \mathcal{P}^{\dagger}\CD$; this is such that
\begin{eqnarray}\label{Pdagger}
[\CD,\mathcal{P}\CC]\cong [\CC,\mathcal{P}^{\dagger}\CD]^{\mathrm{op}}\rlap{ .}
\end{eqnarray}

We write $\mathrm{Mod}$ for the bicategory whose objects are (enriched) categories and whose 
morphisms are modules, also known as ``bimodules'', ``profunctors'' or ``distributors''.
A module $M\dd \CA\tor \CB$ is a functor $M\dd \CA\to \mathcal{P}\CB$ but we
write $M(B,A)$ instead of $(MA)B$. The identity module of $\CB$, also denoted by $\CB$, 
amounts to the Yoneda embedding.

We are assuming the base for enrichment is rich enough that
$\mathrm{Mod}$ should admit right liftings and right extensions
\cite{12, 21}. For modules $M\dd \CA\tor \CB$ and $K\dd \CC\tor \CB$, we
write $\CB(M,K)\dd \CC\tor \CA$ for the right lifting of $K$ through
$M$; it is given on objects by
\begin{eqnarray}\label{lifting}
\CB(M,K)(A,C) = \mathcal{P}\CB(M(-,A),K(-,C)) \ .
\end{eqnarray}

We identify a functor $F\dd \CB\to \CC$ with the module
$F\dd \CB\tor \CC$ given on objects by $F(C,B)=\CC(C,FB)$. With this
notation and that of~\eqref{lifting}, we have that
$F = \CC(1, F) \colon \CB \tor \CC$. On the other hand, we also have
$\CC(F,1)\dd \CC\tor \CB$, which is right adjoint to $\CC(1,F)$ in
the bicategory $\mathrm{Mod}$.

For a module $M\dd \CA\tor \CB$, the {\em $M$-weighted colimit} of a
functor $F\dd \CB\to \CC$, when it exists, is a functor 
$Z\dd \CA\to \CC$ together with an isomorphism
\begin{eqnarray}\label{defweightedcol}
\CC(Z,1)\cong \CB(M,\CC(F,1)) \ .
\end{eqnarray}
By the Yoneda Lemma, such an isomorphism is induced by a module morphism 
$\alpha \dd M\Ra \CC(F,Z)$ which is said to {\em exhibit} the colimit. 
The colimit is {\em preserved} by a functor $H\dd \CC\to \CX$ when the
module morphism 
\begin{eqnarray*}
\CX(HZ,1)\xra{\CX(HF,-)}\CB(\CX(HF,HZ),\CX(HF,1))\xra{\CB(H_{F,Z},1)} \\
\CB(\CC(F,Z),\CX(HF,1))
\xra{\CB(\alpha,1)}\CB(M,\CX(HF,1))
\end{eqnarray*}
is invertible (the case $H=1_{\CC}$ shows how to recover the isomorphism \eqref{defweightedcol}
from $\alpha$).
The $M$-weighted colimit $Z$ is called {\em absolute} when it is preserved by all functors out of $\CC$.

\begin{theorem}\label{MT} 
  Consider a module $M\dd \CA\tor \CB$ and a functor $F\dd \CB\to\CC$.
  The following conditions are equivalent:
\begin{enumerate}[(i),itemsep=0.25\baselineskip]
\item $F$ has an absolute $M$-weighted colimit;
\item There is a functor $Z\dd \CA\to\CC$ and module morphisms ${\alpha
  \dd M\Ra \CC(F,Z)}$ and $\beta \dd Z\Ra F\circ M$ satisfying \eqref{alphabeta1} and \eqref{alphabeta2} below;
\item The composite module $F\circ M \colon \CA \tor \CC$ is representable (that is, $F\circ M\cong Z$ for some functor
$Z\dd \CA\to\CC$);
\item (If idempotents split in $\CC$) There exist a functor $Z\dd \CA\to\CC$ and module morphisms $\alpha \dd M\Ra \CC(F,Z)$, $\beta \dd Z\Ra F\circ M$ satisfying condition \eqref{alphabeta1} below. 
\end{enumerate}
\begin{equation}\label{alphabeta1} 
   \cd[@+0.5em]{&
    {\CA} \ar|@{|}[r]^-{M} \ar|@{|}[d]_{M} \ar|@{|}[dr]^(0.4){Z}  &
    {\CB} \dtwocell[0.3]{dl}{\alpha} \dtwocell[0.7]{dl}{\beta} \twoeq{drrr}& & &
    {\CA} \ar|@{|}[r]^-{M} \ar|@{|}[d]_{M}   &
    {\CB}  \\ &
    {\CB} \ar|@{|}[r]_-{F} &
    {\CC}\ar|@{|}[u]_{\CC(F,1)} &&&
    {\CB} \ar|@{|}[r]_-{F} \ar|@{|}[ur]^-{\CB}&
    {\CC}\ar|@{|}[u]_{\CC(F,1)}\dtwocell[0.35]{ul}{\eta_F}
  }
\end{equation}

\begin{equation}\label{alphabeta2} 
\begin{aligned}
   \cd{
      & & &\CC \\
      \CA \ar|@{|}@/^6pt/[urrr]^{Z} \ar|@{|}[rr]^-{M} \ar|@{|}@/_6pt/[drrr]_{Z} \rtwocell[1.35]{rr}{\varepsilon_F} & & \CB \ar[ur]|-{F} \dtwocell[0.35]{ul}{\beta} \dtwocell[0.35]{dl}{\alpha}  \\
      & & & \CC \ar[ul]|-{\CC(F,1)} \ar|@{|}[uu]_{\CC}
  }
  \qquad {=}
  \qquad
 \cd{
      & & &\CC \\
      \CA \ar|@{|}@/^6pt/[urrr]^{Z}  \ar|@{|}@/_6pt/[drrr]_{Z} \twoeq[.9]{rr} & &   \\
      & & & \CC \ar|@{|}[uu]_{\CC}
  }
    \end{aligned}
\end{equation}
\end{theorem} 

\begin{proof} \emph{(i) $\Ra$ (ii)}: Suppose the module morphism $\alpha \dd
  M\Ra \CC(F,Z)$ exibits $Z \colon \CA \rightarrow \CC$ as
an absolute $M$-weighted colimit of $F$. Form the collage
\begin{equation}
\begin{aligned}
 \cd{
    {\CA} \ar|@{|}[r]^-{M} \ar|@{|}[dr]_-{U}  & {\CB} \ar|@{|}[r]^-{F} & {\CC}\ar|@{|}[dl]^-{V} \\
    & {\CC'} \ltwocello{u}{\gamma} & {}
  }
\end{aligned}
\end{equation}
of the module $F\circ M$.
Herein, $U$ and $V$ are functors and $F \circ M
\cong \CC'(V,U)$ (cf.~\cite{18}).
The mate of $\gamma \colon VFM \Rightarrow U$ under the
adjunction $VF \dashv
{\CC'}(VF,1)$ gives a module morphism
$\bar{\gamma}\dd M \Ra \CC'(VF,1)\circ U \cong \CC'(VF,U)$, while absoluteness
gives $V\circ Z$ as an $M$-weighted colimit of $V\circ F$ exhibited by 
\begin{equation}
  \label{colimit-preserved}
  M\xRa{\alpha}\CC(F,Z)\xRa{V}\CC'(VF,VZ) \ .
\end{equation}
So, by the existence clause in the definition of colimit, there
is a module morphism
$\bar{\beta}\dd \CA \Ra \CC'(VZ,U) \cong \CC'(Z, 1) \circ \CC'(V,U)
\cong \CC'(Z,1) \circ FM$, whose mate $\beta$ under the adjunction
$Z \dashv \CC'(Z,1)$ is a module map $Z \Rightarrow F \circ M$. From
the fact that $\bar \beta$ factors $\bar \gamma$
through~\eqref{colimit-preserved}, it follows that this $\beta$
satisfies \eqref{alphabeta1}; and now
\eqref{alphabeta2} follows using the uniqueness clause in the
definition of colimit.

\emph{(ii) $\Ra$ (iii)}: Let
$\bar \alpha \colon F \circ M \Rightarrow Z$ be the mate of $\alpha$
under the adjunction $F \dashv \CC(F, 1)$. Clearly, \eqref{alphabeta2}
says that $\beta$ and $\bar \alpha$ exhibit $Z$ as a retract of
$F \circ M$. Now by a simple computation, or a lovely argument using
strings, \eqref{alphabeta1} says that the idempotent on $F\circ M$
coming from the retraction is the identity.

\emph{(iv) $\Ra$ (ii)}: Much as in the last case, we still obtain an
idempotent on $Z$ using \eqref{alphabeta1}. Split the idempotent on
$Z$ to obtain $Z'$, $\alpha'$, $\beta'$ as in condition (ii).

\emph{(ii) $\Ra$ (iv)} is obvious.

\emph{(iii) $\Ra$ (i)}: Clearly condition (iii) is preserved by any
functor out of $\CC$ so all that remains is to see that it implies $Z$
is an $M$-weighted colimit of $F$. However, since $F \circ M \cong Z$
by (iii), we have
\begin{eqnarray*}
\CC(Z,1)\cong \CC(F\circ M,1)\cong \CB(M,\CC(F,1)) \ ,  
\end{eqnarray*}
as in \eqref{defweightedcol}. 
\end{proof}

To obtain the result for limits rather than colimits, notice that the limit of the functor $F\dd \CB\to \CC$ weighted by a module $N\dd \CB\tor \CA$ is the colimit of $F\dd \CB^{\mathrm{op}}\to \CC^{\mathrm{op}}$
weighted by $N^{\mathrm{op}}\dd \CA^{\mathrm{op}}\tor \CB^{\mathrm{op}}$.

\section{Par\'e $\Rightarrow$ Street $\wedge$ Garner}\label{allfun}

The second author showed in~\cite{23} that, when a module
$M \colon \CA \tor \CB$ has a right adjoint $N$, \emph{all}
$M$-weighted colimits are necessarily absolute. In~\cite{Garn2014},
the first author refined this by showing that, in this situation,
making a functor $F \colon \CB \rightarrow \CC$ into an $M$-weighted
colimit for $F$ is equivalent to making it into an $N$-weighted
limit, and that both are in turn the same as providing:
\begin{itemize}
\item A module morphism $\alpha \colon M \Rightarrow \CC(F,Z)$; and
\item A module morphism $\bar\beta \colon N \Rightarrow \CC(Z,F)$,
\end{itemize}
subject to the following two conditions:
\begin{equation}\label{alphabeta3} 
   \cd[@+0.5em]{&
    {\CA} \ar|@{|}[rr]^-{M} \ar|@{|}@/^6pt/[drr]^-{Z}  & &
    {\CB} \dtwocell[0.3]{dll}{\alpha}  \twoeq{drrr}& & &
    {\CA} \ar|@{|}[r]^-{M}  &
    {\CB}  \\ &
    {\CB} \ar|@{|}[u]^{N} \rtwocello[0.3]{ur}{\!\!\bar\beta} \ar|@{|}[r]_-{F} &
    {\CC} \ar|@{|}[r]_-{\CC} \dtwocell[0.35]{u}{\varepsilon_Z} \ar[ul]|-{\CC(Z,1)} &
    {\CC}\ar|@{|}[u]_{\CC(F,1)} &&&
    {\CB} \ar|@{|}[r]_-{F} \ar[ur]|-{\CB} \ar|@{|}[u]^-{N} &
    {\CC}\ar|@{|}[u]_{\CC(F,1)}\dtwocell[0.3]{ul}{\eta_F}\dtwocell[0.7]{ul}{\varepsilon}
    }
\end{equation}

\begin{equation}\label{alphabeta4} 
\begin{aligned}
   \cd{
      & \CA \rtwocell[0.6]{d}{\eta}  & &\CC \ar|@{|}[ll]_-{\CC(Z,1)}\\
      \CA \ar|@{|}[ur]^{\CA} \ar|@{|}[rr]^-{M} \ar|@{|}@/_6pt/[drrr]_{Z} \rtwocell[1.35]{rr}{\varepsilon_F} & & \CB \ar|@{|}[ur]^-{F} \ar|@{|}[ul]_-{N}  \rtwocell{u}{\bar\beta} \dtwocell[0.35]{dl}{\alpha}  \\
      & & & \CC \ar[ul]|-{\CC(F,1)} \ar|@{|}[uu]_{\CC}
  }
  \qquad {=}
  \qquad
 \cd{
      & \CA \rtwocell[0.4]{ddrr}{\eta_Z}  & &\CC \ar|@{|}[ll]_-{\CC(Z,1)}\\
      \CA \ar|@{|}[ur]^{\CA} \ar|@{|}@/_6pt/[drrr]_{Z} \\
      & & & \CC\rlap{ .} \ar|@{|}[uu]_{\CC}
  }
    \end{aligned}
\end{equation}

This result is in fact an easy consequence of our Theorem~\ref{MT}.
Indeed, given $\alpha$ and $\bar \beta$ as above, taking the mate of
$\bar \beta \colon N \rightarrow \CC(Z,F)$ under the adjunctions
$Z \dashv \CC(Z, 1)$ and $M \dashv N$ yields a module morphism
$\beta \colon Z \Rightarrow F \circ M$, with the
conditions~\eqref{alphabeta3} and~\eqref{alphabeta4} transforming
under the mates correspondence to the conditions~\eqref{alphabeta1}
and~\eqref{alphabeta2}. So in this situation, $\alpha$ exhibits $Z$ as
the $M$-weighted colimit of $F$, and this colimit is necessarily
absolute. Because the data above are self-dual (that is, correspond to an
instance of the same data in $\CC^\mathrm{op}$), it likewise follows
that $\bar \beta$ exhibits $Z$ as the $N$-weighted limit of $F$, and
this limit is necessarily absolute.

A similar argument also reconstructs the main result of the
second author's~\cite{23}, namely that colimits by left adjoint modules are 
necessarily absolute; for completeness, we now provide a result which
collates various known characterisations of absolute weights; these can
be extracted from Theorem I.4.1 of Dubuc \cite{DubucPhD}, Theoreme
I.1.1 of Gouzou-Grunig \cite{GouGru}, Proposition 1 of Johnson's
\cite{SRJ2}, Chapter 4 Section 1 Exercise 4 (Par\'e) of Mac Lane
\cite{CWM}, and the aforementioned main result of~\cite{23}.
The equivalence of items (i) and (iii) was also established by Arkor-McDermott
as Lemma 2.32 of \cite{ArkMcD} showing it to hold at the generality of $j$-absolute colimits.
Compare also Theorem 5.16(c--d) of \cite{Kou}.
The reader might read \cite{KelSch} for background on various classes of colimits, particularly see Section 6 of that paper. 

\begin{theorem}\label{reconstr}
The following conditions on a module $M\dd \CA\tor \CB$ are equivalent:
\begin{enumerate}[(i),itemsep=0.25\baselineskip]
\item $M$ is a left adjoint in $\mathrm{Mod}$;
\item Every $M$-weighted colimit is absolute;
\item The $M$-weighted colimit of $\mathrm{Y}^{\dagger}\dd \CB\to \mathcal{P}^{\dagger}\CB$
is absolute;
\item If $R\dd \CB\tor \CA$ is the right lifting of the identity module
  of $\CB$ along $M$, exhibited by the 2-cell
  $\varepsilon\dd M\circ R\Ra \CB(1,1)$, then there exists a module
  morphism $\eta \dd \CA(1,1)\Ra R\circ M$ such that
  $(\varepsilon \circ M)(M\circ \eta) = 1_M$;
\item For each object $U$ of the base bicategory and each object
  $A\dd U\to \CA$ of $\CA$, the module $M(-,A)\dd U\tor \CB$ has a
  right adjoint module.
\end{enumerate}
\end{theorem} 
\begin{proof}
  \emph{(i) $\Ra$ (ii)}: Assume $M\dashv N$ and that $Z\dd \CA\to \CC$ is
  the $M$-weighted colimit of $F\dd \CB\to \CC$. Using
  \eqref{defweightedcol}, we see that
  $\CC(Z,1)\cong \CA(1,N\circ F) \cong N\circ \CC(F,1)$. Taking left
  adjoints yields $Z\cong F\circ M$ and the result follows by
  Theorem~\ref{MT}(iii).

  \emph{(ii) $\Ra$ (iii)} is obvious.

  \emph{(iii) $\Ra$ (i)}: Let $Z\dd \CA\to \mathcal{P}^{\dagger}\CB$
  be the $M$-weighted colimit of $\mathrm{Y}^{\dagger}$, and let
  $N\dd \CB\tor \CA$ be the module corresponding to the functor $Z$
  under the isomorphism \eqref{Pdagger}; the formula is
  $N(A,B)=\mathrm{lim}(M(-,A),\CB(1,B))$. We shall show that $M$ is a
  left adjoint for $N$, or equivalently, that $M$ is the absolute
  right extension of $\CB\dd \CB\tor \CB$ along $N$. In other words,
  we must show that, for all modules $T\dd \CB\tor \CX$, the module
  $T\circ M$ is the right extension of $T$ along $N$. Notice that,
  using Yoneda, we have
$$(T\circ M)(-,A) \cong \mathrm{colim}(M(-,A),\mathcal{P}^{\dagger}\CB(T,\mathrm{Y}^{\dagger})) 
\cong (\mathrm{Y}^{\dagger}\circ M)(T,A) \ .$$ By
Theorem~\ref{MT}(iii), $\mathrm{Y}^{\dagger}\circ M\cong Z$. However,
$Z(T,A) \cong \mathrm{lim}(N(A,-),T)$ showing that
$Z(T,-) \cong T \circ M$ is the right extension of $T$ along $N$.

\emph{(i) $\Leftrightarrow$ (iv)}: The $\Ra$ is obvious. In the other
direction, the data of (iv) are a candidate for unit and counit for
$R$ as right adjoint to $M$. However, only one of the adjunction
identities is assumed to hold. By Par\'e's Exercise in \cite{CWM}, the
right adjoint to $M$ is obtained by splitting an idempotent on $R$.

\emph{(i) $\Leftrightarrow$ (v)}: The $\Ra$ is again obvious. In the
other direction, let $\CU$ be the category whose objects are those of
$\CA$ and whose homs $\CU(A,B)$ are all initial objects unless $A=B$
in which case the hom is an identity morphism in the base bicategory.
So the inclusion $J\dd \CU\to \CA$ is opmonadic (= ``of Kleisli
type'') in $\mathrm{Mod}$; see \cite{18}. By a dual to Dubuc's Adjoint
Triangle Theorem \cite{DubucAT}, it follows that $M$ has a right
adjoint if $M\circ J$ does, and the result follows.
\end{proof}

In the case of enrichment over the monoidal poset of extended positive
real numbers $(\mathbb{R}_+ \cup \{\infty\}, \geqslant)$, the results
of~\cite{LawMetric} show that enriched categories are
generalised metric spaces, and the absolute weights are generated by
those for limits of Cauchy sequences. Thus, it is common to refer to
modules satisfying the equivalent conditions of~Theorem
\ref{reconstr} as \emph{Cauchy modules} or \emph{Cauchy weights}.

\section{Street $\Rightarrow$ Par\'e}
In this section, we obtain a kind of converse to Theorem~\ref{MT},
showing that any absolute colimit is in fact a colimit by an
absolute weight. More precisely, let us consider a module
$M\dd \CA\tor \CB$ and a functor $F\dd \CB\to\CC$, and suppose that
\begin{equation}\label{eq:firstcolim}
  \alpha \colon M \Rightarrow \CC(F,Z)
\end{equation}
exhibits $Z \colon \CA \rightarrow \CC$ as an $M$-weighted
colimit of $F$. What we will do is to replace $M$ by
$M'$, $F$ by $F'$, and $\alpha$ by a suitable
\begin{equation}\label{eq:secondcolim}
  \alpha' \colon M' \Rightarrow \CC(F',Z)
\end{equation}
exhibiting $Z$ as a $M'$-weighted colimit of $F'$, in
such a way that if the \emph{colimit} exhibited by~\eqref{eq:firstcolim}
is absolute, then the \emph{weight} $M'$ for the
colimit~\eqref{eq:secondcolim} is absolute, i.e., $M'$ is a Cauchy
module. In this way, the apparently more general absolute
colimits of~\cite{Par1969} and Theorem~\ref{MT} reduce to the apparently less
general absolute colimits of~\cite{23}, \cite{Garn2014} and
Theorem~\ref{reconstr}.

To this end, we factorise the functor $F \colon \CB
\rightarrow \CC$
as to the left in:
\begin{equation}
  \label{eq:factorisation}
  \cd{
    F = \CB \ar[r]^-{J} & \CB' \ar[r]^-{F'} & \CC & & \CA \ar|@{|}[r]^-{M} & \CB\ar|@{|}[r]^-{J} &  \CB'\rlap{ .}
  }
\end{equation}
where $J$ is the identity on objects, and $F'$ is fully faithful, and
define $M'$ to be the composite module as to the right. We now obtain
the $\alpha'$ of~\eqref{eq:secondcolim} by transposing the composite
\begin{equation*}
  M \xRa{\ \alpha\ } \CC(F,Z) \xRa{\ \cong\ } \CC(J,1) \circ \CC(F',Z)
\end{equation*}
under the adjointness $J \dashv \CC(J,1)$. By a straightforward and
standard manipulation in~\eqref{defweightedcol}, $\alpha'$
exhibits $Z$ as the $M'$-weighted colimit of $F'$. 
  
\begin{theorem}\label{MT2} 
  In the situation just described, if $\alpha$ exhibits $Z$ as an
  \emph{absolute} $M$-weighted colimit of $F$, then $M'$
  is a Cauchy module.
\end{theorem}
\begin{proof}
  If the colimit exhibited by~\eqref{eq:firstcolim} is absolute, then
  by (i) $\Rightarrow$ (iii) of Theorem~\ref{MT}, we have that
  $F \circ M  = F' \circ JM \colon \CA \tor \CC$ is representable, and so in
  particular has a right adjoint in $\mathrm{Mod}$. Furthermore
  $F' = \CC(1,F') \colon \CB' \tor \CC$ has the right adjoint
  $\CC(F',1)$ in $\mathrm{Mod}$, and the unit of the adjunction:
  \begin{equation*}
    \CC(1,1) \Rightarrow \CC(1,F') \circ \CC(F',1) \cong \CC(F',F')\rlap{ .}
  \end{equation*}
  is
  invertible since $F'$ is fully faithful. So we have a pair of
  composable morphisms
  \begin{equation*}
    \cd{
      \CA \ar|@{|}[r]^-{JM} & \CB' \ar|@{|}[r]^-{F'} & \CC
    }
  \end{equation*}
  in $\mathrm{Mod}$ for which the composite has a right adjoint, and
  the second arrow has a right adjoint with invertible unit. As it
  would in any bicategory, this implies that $JM = J'$ has a right
  adjoint, found as the composite of $F'$ with the right
  adjoint $\CC \tor \CA$. So $J'$ is Cauchy as claimed.
\end{proof}

\section{Examples}

\begin{example}\label{abscoeq}
  Let $\CV = \mathrm{Set}$, let $\CA = \mathbf{1}$ (the terminal
  category), let $\CB$ be the free category on the graph with two
  vertices $0, 1$ and two edges $1 \rightrightarrows 0$, and let
  $M \colon \mathbf{1} \tor \CB$ be the terminal module. In this case,
  if $F \colon \CB\to \CC$ takes the generating graph to the parallel
  pair $u, v \colon X \rightrightarrows Y$ in $\CC$, then a module
  morphism $\alpha \colon M \Rightarrow \CC(F,Z)$ involves maps
  $\alpha_1 \colon X \rightarrow Z$ and
  $\alpha_0 \colon Y \rightarrow Z$ with
  $\alpha_0 u = \alpha_1 = \alpha_0 v$. Putting $a = \alpha_0$, this
  is just a map $a \colon Y \rightarrow Z$ with $au=av$, i.e., a
  cofork under $(u,v)$: so an $M$-weighted colimit of $F$ is a
  coequaliser of $u$~and~$v$.

  Now look at Theorem~\ref{MT}(ii). The module morphism
  $\alpha \colon M \Rightarrow \CC(F,Z)$ is a cofork
  $a \colon Y \rightarrow Z$ under $(u,v)$, and by Yoneda, the module
  morphism $\beta \colon Z \Rightarrow F \circ M$ amounts to an
  element
  $\beta \in (F \circ M)(Z) = \mathrm{colim}_{B \in \CB}\CC(Z,FB)$.
  Now \eqref{alphabeta2} says that $\beta$ has a representative
  $b\in \CC(Z,Y)$ with $ab = 1_Z$; while \eqref{alphabeta1} is the
  requirement that $ba$ and $1_Y$ are identified in the coequaliser of
  the parallel pair
  $\CC(Y,u), \CC(Y,v) \colon \CC(Y,X) \rightrightarrows \CC(Y,Y)$ in
  $\mathrm{Set}$. By the usual construction of coequalisers in
  $\mathrm{Set}$, we thus re-find Proposition~5.3 of~\cite{Par1969},
  characterising the data required to make a coequaliser absolute as being:
  \begin{enumerate}[(i)]
  \item A map $a \colon Y \rightarrow Z$ with $au = av$;
  \item A map $b \colon Z \rightarrow Y$ with $ab = 1_Z$;
  \item Maps $y_0, \dots, y_n \colon Y \rightarrow Y$ and $x_1, \dots,
    x_n \colon Y \rightarrow X$ with $y_0 = 1_Y$ and $y_n = ba$ and,
    for each $1 \leqslant i \leqslant n$,
    either $ux_i = y_{i-1}$ and $vx_i = y_i$, or $ux_i = y_{i}$ and $vx_i = y_{i-1}$.
  \end{enumerate}

  Let us also examine the force of Theorem~\ref{MT2} in this
  situation. For simplicity, let us assume that $X \neq Y$ and $u \neq
  v$: then the category $\CB'$ of~\eqref{eq:factorisation} is the full subcategory of
  $\CC$ on the objects $X$ and $Y$, the diagram
  $F' \colon \CB' \rightarrow \CC$ is the inclusion functor, and the
  weight $M'$, which is no longer terminal, instructs us to throw away
  all of the diagram $F'$ except for the arrows $u$ and $v$, and then
  take the coequaliser.

  According to Theorem~\ref{MT2}, the weight $M'$ in this situation is
  \emph{absolute}. This is to say that, if we take a $\CB'$-shaped
  diagram in any category $\CD$, throw away everything except the
  arrows corresponding to $u$ and $v$, and take the coequaliser, then
  this coequaliser will be absolute. The reason for this is that the
  maps in (iii) above, exhibiting $a \colon Y \rightarrow Z$ as an
  absolute coequaliser of $u$ and $v$, live in $\CB'$ and their images
  in $\CD$ suffice to make the coequaliser \emph{there} absolute.

  To make this a little more concrete, consider the case of (i)--(iii)
  above where $n=1$: this yields the familiar notion of a
  \emph{split coequaliser}, wherein the fork $au = av \colon X \rightrightarrows Y
  \rightarrow Z$ is split by maps $b \colon Z \rightarrow Y$ and $x
  \colon Y \rightarrow X$ such that $ab = 1_Z$, $ux = 1_Y$ and $vx =
  ba$.
  In this situation, the category $\CB'$ contains, among other things,
  the parallel pair $u,v \colon X \rightrightarrows Y$ and the map
  $x \colon Y \rightarrow X$, satisfying $ux = 1_Y$ and $vxu =
  vxv$---so that the image of the parallel pair $u,v$ under any
  functor $\CB' \rightarrow \CD$ will be what is sometimes called a
  \emph{contractible pair}---whose coequaliser is equally a splitting
  of the idempotent $vx$, whence absolute.

  Of course, in this situation the category $\CB'$ will typically
  contain many other things than the data of a contractible pair. However,
  if the ``reason'' for the absoluteness of our original coequaliser had
  involved a case of (iii) with $n > 1$, then some of this other data
  would have been needed to force absoluteness of the weight $M'$.
  Thus, our choice of $\CB'$ is simply a blunt instrument which is
  known to work in all cases.

  Note that, in particular, $\CB'$ could well be of arbitrary (small)
  cardinality, even though $\CB$ itself is resolutely finite. It is
  apparent from our more detailed analysis that, for any particular
  absolute coequaliser, we can take instead of $\CB'$ a suitable some
  \emph{finite} extension of $\CB$ wherein the weight $M$ becomes
  absolute, thereby expressing any absolute coequaliser as a colimit
  by a finite absolute weight. It is unclear if this is a peculiar
  feature of this example, or if something similar is true more
  generally.
\end{example}

\begin{example}
  Take $\CV = \mathrm{Cat}$, so that $\CV$-categories are
  $2$-categories. As in~\cite{Schanuel1986The-free}, we have the
  $2$-category $\mathrm{Adj}$ freely generated by objects $X$ and $Y$
  and an adjunction $f \dashv u \colon X \rightarrow Y$. Let
  $\CA = \mathbf{1}$, the terminal $2$-category, let $\CB$ be the full
  sub-$2$-category $I \colon \CB \hookrightarrow \mathrm{Adj}$ of
  $\mathrm{Adj}$ on the object $Y$, and let
  $M = \mathrm{Adj}(I,X) \colon 1 \tor \CB$. In this case $\CB$ is the
  suspension of $\Delta_+$, the free strict monoidal category
  containing a monoid, so that a $\CV$-functor
  $F \colon \CB \rightarrow \CC$ is a monad $(T,t)$ in $\CC$. A module
  morphism $\alpha \colon M \Rightarrow \CC(F,Z)$ is now given by a
  $1$-cell $a \colon T \rightarrow Z$ together with a $2$-cell
  $\varphi \colon at \Rightarrow a$ satisfying associativity and
  unitality laws; we may speak of a \emph{right $t$-module}. An
  $M$-weighted colimit for $F$ is a universal right $t$-module, so a
  \emph{Kleisli object} for $(T,t)$; see~\cite{Street1972The-formal}.

  Consider Theorem~\ref{MT}(ii). Now $\alpha$ is a right $t$-module
  $(a \colon T \rightarrow Z, \varphi \colon at \Rightarrow a)$, while
  by Yoneda, $\beta \colon Z \Rightarrow F \circ M$ is an object of
  $(F \circ M)(Z) = \CC(Z, T)_{\CC(Z,t)}$, the Kleisli category in
  $\mathrm{Cat}$ of the monad
  $\CC(Z,t) \colon \CC(Z,T) \rightarrow \CC(Z,T)$, so equally an
  object $b \in \CC(Z,T)$. Now, condition~\eqref{alphabeta2} tells us
  like before that $ab = 1_Z$; while condition~\eqref{alphabeta1}
  makes the two assertions that:
  \begin{itemize}
  \item $ba = 1_T$ in $\CC(T,T)_{\CC(T,t)}$, that is, $b$ and $a$ are
    inverse isomorphisms;
  \item The image of the map $ba \varphi \colon bat \Rightarrow ba$ of
    $\CC(T,T)$ in $\CC(T,T)_{\CC(T,t)}$ is equal to the universal map
    $bat \Rightarrow ba$, which is to say that
    \begin{equation*}
      1_{bat} = bat \xRa{b\varphi} ba \xRa{ba\eta} bat\rlap{ .}
    \end{equation*}
  \end{itemize}
  But since 
$a$ is a right $t$-module, we also have that
  $b\varphi \circ ba\eta = 1_{ba}$. Thus $ba\eta = \eta$ is an
  invertible $2$-cell, and so the monad $(T,t)$ must be trivial, that is,
  isomorphic to the identity monad on $T$. Thus, a Kleisli object is
  absolute in a $2$-category if and only if it is the Kleisli object
  of a trivial monad.
\end{example}
\begin{remark}
  The preceding result is rather like the situation with coproducts in
  ordinary categories: a coproduct is absolute if and only if it is a
  trivial---which is to say, unary---coproduct. Of course, this result
  is sensitive to the enriching base: in a category enriched over
  commutative monoids, all finite coproducts are absolute, while in a
  category enriched over complete join-lattices, \emph{all} coproducts
  are absolute. In a similar way, the situation with Kleisli objects
  changes completely if we consider enrichment over categories with
  reflexive coequalisers: in that context, \emph{every} Kleisli object
  is absolute~\cite{18}.
\end{remark}

\begin{example}
  Take $\CV = \mathrm{Cat}$, take $\CA = \mathbf{1}$ as before, and
  take $\CB$ to be the locally discrete $2$-category to the left in
  \begin{equation*}
    \cd{
      {2} \ar@<6pt>[r]^{p } \ar[r]|{m} \ar@<-6pt>[r]_{q} &
      {1} \ar@<6pt>[r]^{d} \ar@{<-}[r]|{i} \ar@<-6pt>[r]_{c} &
      {0} 
    } \qquad \cd{
      {W} \ar@<6pt>[r]^{p } \ar[r]|{m} \ar@<-6pt>[r]_{q} &
      {X} \ar@<6pt>[r]^{d} \ar@{<-}[r]|{i} \ar@<-6pt>[r]_{c} &
      {Y} 
    }\qquad 
    \cd{
      {\mathbf{3}} \ar@<6pt>@{<-}[r]^{01} \ar@{<-}[r]|{02} \ar@{<-}@<-6pt>[r]_{12} &
      {\mathbf{2}} \ar@<6pt>@{<-}[r]^{0} \ar[r]|{0} \ar@{<-}@<-6pt>[r]_{1} &
      {\mathbf{1}} 
    }
  \end{equation*}
  subject to the simplicial identities $di = ci = 1_0$, $dp = dm$,
  $cm = cq$ and $cp = dq$. We denote the image of a $2$-functor
  $F \colon \CB \rightarrow \CC$ as centre above. We consider
  $M \colon \mathbf{1} \tor \CB$ the module whose values are displayed
  right above. Here, $\mathbf{2}$ and $\mathbf{3}$ are the posetal
  categories $0\leqslant 1$ and $0 \leqslant 1 \leqslant 2$, and the
  functors are named by giving their images.

  A module morphism $\alpha \colon M \Rightarrow \CC(F,Z)$ is a
  \emph{codescent cocone} under $F$, whose data are the $1$-cell
  $a = \alpha_0(0) \colon Y \rightarrow Z$ and $2$-cell
  $\theta = \alpha_1(0 \leqslant 1) \colon f \circ d \Rightarrow f
  \circ c$, subject to the conditions that $\theta i = 1_f$ and $\theta
  m = \theta q \circ \theta p$.
  An $M$-weighted colimit of $F$ is a universal codescent cocone, also
  termed a \emph{codescent object} for $F$.

  When $\CC = \mathrm{Cat}$, the codescent object of a diagram
  $F \colon \CB \rightarrow \mathrm{Cat}$ is the functor $A_F \colon Y
  \to \mathrm{codesc}(F)$ obtained by adjoining to $Y$ a family of morphisms
  $(\Theta_x \colon d(x) \rightarrow c(x))_{x \in X}$ (which will
  be the components of $\Theta \colon A \circ d \Rightarrow A \circ c$)
  subject to the following identities:
  \begin{enumerate}[(i)]
  \item $c(f) \circ \Theta_{x} = \Theta_{x'} \circ d(f)$ for all
    $f \colon x \rightarrow x'$ in $X$;
  \item $\Theta_{i(y)} = 1_{y}$ for all $y \in Y$;
  \item $\Theta_{m(w)} = \Theta_{q(w)} \circ \Theta_{p(w)}$ for
    all $w \in W$.
  \end{enumerate}
  Note, in particular, that $A$ is the identity on objects.
  
  Now, for a general $F \colon \CB \rightarrow \CC$, we consider once
  again Theorem~\ref{MT}(ii). The module morphism
  $\alpha \colon M \Rightarrow \CC(F,Z)$ is a codescent cocone
  $(a, \theta)$, while $\beta \colon Z \Rightarrow F \circ M$ is an
  object
  $\beta \in (F \circ M)(Z) = \mathrm{codesc}(\CC(Z,F\text{--}))$, so
  equally an object of $\CC(Z,Y)$. Once again, \eqref{alphabeta2} says
  that $ab = 1_Z$; while \eqref{alphabeta1} becomes the two
  requirements that:
  \begin{itemize}
  \item $ba$ and $1_{Y}$ are identified by 
    $A_{\CC(Y,F)} \colon \CC(Y,Y) \rightarrow \mathrm{codesc}(\CC(Y, F))$. Since
    this functor is bijective on objects, this just says that
    $ba = 1_{Y}$.
  \item The morphism  of $\CC(X,Y)$ given by:
    \begin{equation*}
          \cd[@-0.5em]{
      & Y \ar[dr]^-{a} \ar@/^9pt/@{=}[drr] \\
      X \ar[ur]^-{d} \ar[dr]_-{c} \dtwocell{rr}{\theta} & & Z \ar[r]^-{b} & Y\\
      & Y \ar[ur]_-{a} \ar@/_9pt/@{=}[urr]
    }
  \end{equation*}
  is identified by $A_{\CC(X,F)} \colon \CC(X,Y)
  \rightarrow \mathrm{codesc}(\CC(X, F))$ with the
  universally adjoined morphism
  $\Theta_{\mathrm{id}_X} \colon d \rightarrow c$.
  \end{itemize}
  Of course, the first requirement tells us that $a$ is an isomorphism
  with inverse $b$. However, the second requirement is quite
  non-trivial, due to the interactions of the identities (i)--(iii)
  for a codescent object in $\mathrm{Cat}$, and a complete
  characterisation in the spirit of~\cite{Par1969} would be both
  unpleasant and not particularly enlightening. In lieu of this, we
  provide a few \emph{sufficient} conditions for a codescent cocone to
  be absolute, which should illuminate the scope of the possibilities.

  Since we know already that $a$ is invertible, we may as well assume
  it to be the identity: thus, our codescent cocone under $F$ is specified by a
  single $2$-cell $\theta \colon d \Rightarrow c \colon Y \rightarrow X$ satisfying $\theta i
  = 1_{1_X}$ and $\theta m = \theta q \circ \theta p$.
  \begin{itemize}
  \item Suppose there is a $2$-cell
    $\gamma \colon i \circ d \Rightarrow 1_X \colon X \rightarrow X$
    such that $d \circ \gamma = 1_d$ and $c \circ \gamma = \theta$.
    Then the codescent cocone is absolute, since in
    $\mathrm{codesc}(\CC(Y,F))$ we have that
    \begin{equation*}
      \Theta_{1_X} = \Theta_{1_X} \circ 1_d = \Theta_{1_X} \circ d(\gamma) = c(\gamma) \circ \Theta_{i(d)} = c \circ \gamma = \theta \rlap{ .}
    \end{equation*}
  \item Dually, the codescent cocone is absolute if there is $\delta
    \colon 1_X \Rightarrow ic$ with $d\delta = \theta$ and $c\delta =
    1_c$.
  \item Suppose there is $k \colon X \Rightarrow W$ for which
    $mk = 1_X$, along with $2$-cells $\gamma \colon id \Rightarrow pk
    \colon X \rightarrow X$
    and $\delta \colon qk \Rightarrow ic \colon X \rightarrow X$ for
    which
    \begin{equation*}
      d \gamma = 1_d \quad c \delta = 1_c \quad \text{and} \quad d \delta \circ c \gamma = \theta\rlap{ .}
    \end{equation*}
    Then the codescent cocone is absolute, since:
    \begin{align*}
      \Theta_{1_X} &= \Theta_{mk} = \Theta_{qk} \circ \Theta_{pk} = c \delta \circ \Theta_{qk} \circ \Theta_{pk} \circ d \gamma \\ &= \Theta_{ic} \circ d \delta \circ c \gamma \circ \Theta_{id} = d\delta \circ c\gamma = \theta\rlap{ .}
    \end{align*}
  \item Suppose there is a $1$-cell $g \colon X \rightarrow X$ such
    that $dg = d$ and $cg = c$, along with $2$-cells
    $\gamma \colon i \circ d \Rightarrow g$ and $\delta \colon 1_X
    \Rightarrow g \colon X \rightarrow X$
    such that
    \begin{equation*}
      d \gamma = 1_d \qquad c\gamma = \theta \quad  d\delta = 1_d \quad \text{and} \quad  c\delta = 1_c\rlap{ .}
    \end{equation*}
    Then the codescent cocone is
    absolute, since:
    \begin{equation*}
      \Theta_{1_X} = c\delta \circ \Theta_{1_X} = \Theta_{g} \circ d\delta = \Theta_g \circ d\gamma = c\gamma \circ \Theta_{id} = c\gamma = \theta\rlap{ .} 
    \end{equation*}
  \end{itemize}
  The reader is encouraged to explore other examples combining the
  techniques illustrated above.
\end{example}

\section{Remarks on Cauchy weights}
\label{sec:remarks-cauchy-weigh}

To conclude the paper, we discuss some important necessary conditions
to be a Cauchy weight, and circumstances under which these are in fact
sufficient. In this section, we assume our base for enrichment $\CV$
is a monoidal category (and not a bicategory.)

\begin{proposition}\label{landsinduals}
  Let $\CV$ be locally finitely presentable as a closed
  category~\cite{Kelly1982Structures}, and let $\CC$ be a
  $\CV$-category with finitely presentable homs. If the module
  $M \in [\CC^{\mathrm{op}},\CV]$ is finitely presentable---and in
  particular, if it is Cauchy---then it lands in the subcategory
  $\CV_{\mathrm{fp}}$ of $\CV$ consisting of the finitely presentable
  objects.
\end{proposition}
Here, we note that if $M \in [\CC^\mathrm{op}, \CV]$ is Cauchy with
right adjoint $M^\ast \in [\CC, \CV]$, then 
$[\CC^\mathrm{op}, \CV](M, \text{--}) \cong M^\ast \otimes (\text{--}) \colon [\CC^\mathrm{op}, \CV]
\rightarrow \CV$ will preserve small limits, and so in particular
filtered colimits; whence Cauchy implies finitely presentable.
\begin{proof}
For such a finitely presentable $M$ and any $A\in \CC$, we must show that $[MA, -] \dd \CV\to \CV$ preserves filtered colimits.
We have isomorphisms
\begin{eqnarray*}
[MA,V] & \cong & [\int^B MB\ \ot \ \CC(A,B),V] \\
& \cong & [\CC^{\mathrm{op}}, \CV](M, [\CC(A,-),V])
\end{eqnarray*}
both $\CV$-natural in $V$.
Since each $\CC(A,B)$ is finitely presentable, $[\CC(A,B), V]$ preserves filtered colimits in the variable $V$. 
Since colimits in $[\CC^{\mathrm{op}}, \CV]$ are formed pointwise,  $[\CC(A,-), V]$ preserves filtered colimits
in the variable $V$. Then $M$ finitely presentable implies $[\CC^{\mathrm{op}}, \CV](M, [\CC(A,-),V])$ preserves filtered colimits
in the variable $V$.  
\end{proof}
In the case that $\CV$ is not just locally finitely presentable as a
closed category, but ``locally $0$-presentable''---i.e.,
$\CV = [\CC^\mathrm{op}, \mathrm{Set}]$ under the convolution monoidal
structure coming from a symmetric monoidal structure on $\CC$---an
analogous proof shows that: if $\CC$ is a $\CV$-category with
small-projective homs, then any small-projective (= Cauchy)
$M \in [\CC^\mathrm{op}, \CV]$ takes values in the subcategory
$\CV_{\mathrm{sp}}$ of small-projective objects.

\begin{proposition}\label{ppzero}
Let $\CV$ be locally finitely presentable as a closed category. Let $\CC$ be a small
$\CV$-category which can be written as
\begin{eqnarray*}
  \CC = \bigcup_{i} \CC_i
\end{eqnarray*}
where $\{\CC_i : i \in I\}$ is a directed family of full
sub-$\CV$-categories, each of which is a cosieve in $\CC$, i.e.,
$\CC(W,U)=0$ for $U\in \CC_i$ and $W\notin \CC_i$. If
$M\in [\CC^{\mathrm{op}}, \CV]$ is finitely presentable then it is the
Kan extension of its own restriction to some $\CC_i$. In particular,
this is the case if $M \colon \CI\tor \CC$ is a Cauchy module.
\end{proposition}
\begin{proof}
Let $M^{(i)}\in [\CC^{\mathrm{op}}, \CV]$ be the Kan extension of the restriction of $M$ to 
the full sub-$\CV$-category $\CC_i$; thus we have the following, where
the second equality follows since $\CC_i$ is a cosieve in
$\CC$:
\begin{equation*}
  M^{(i)}W = \int^{U\in \CC_i}\CC(W,U)\ot MU =
  \begin{cases}
    MW & \text{ if } W \in \CC_i\text{,} \\ 0 & \text{ otherwise.}
  \end{cases}
\end{equation*}
Clearly, $M$ is the directed colimit of the $M^{(i)}$'s.
Since $[\CC^{\mathrm{op}}, \CV](M,-)$ preserves filtered colimits, we have
\begin{eqnarray*}
[\CC^{\mathrm{op}}, \CV](M,M) \cong \mathrm{colim}_i[\CC^{\mathrm{op}}, \CV](M,M^{(i)}) \ .
\end{eqnarray*}
This is true in $\CV$ at the enriched level, but also in the underlying category of $\CV$.
So each endomorphism of $M$ factors through some $M^{(n)}$. In particular, there exists $n$ such that 
\begin{eqnarray*}
\xymatrix{
M \ar[rd]_{}\ar[rr]^{1_M}    && M  \\
& M^{(n)}  \ar[ru]_{\mathrm{in}_n} &
}
\end{eqnarray*}
commutes.
Of course, $\mathrm{in}_n$ has invertible component at any $W \in
\CC_n$. But for $W\notin \CC_n$, we have the diagram
\begin{eqnarray*}
\xymatrix{
& MW \ar[rd]_{}\ar[rr]^{1_{MW}}    && MW  \\
0 \ar[rr]^{1_0} \ar[ru]^{!} & & 0  \ar[ru]_{!} & 
}\rlap{ .}
\end{eqnarray*}  
which shows $\mathrm{in}_n$ is also invertible there, and so $M^{(n)}
\cong M$.
\end{proof}

In some important cases, Propositions~\ref{landsinduals} and \ref{ppzero} provide the sufficient conditions for 
an $M$ to be Cauchy. Compare the next result with Lemma 4 of \cite{BMT}.

\begin{proposition}\label{singlerep}
  Suppose $\CV$ is the monoidal category of vector spaces over a field
  $k$ of characteristic zero. Let $\CC=k\CG$ be the free
  $\CV$-category on a groupoid $\CG$ with finite homs. A functor
  $M \colon \CC^\mathrm{op} \rightarrow \CV$ is Cauchy as a module
  $\CI\tor \CC$ if and only if $Mx$ is finite dimensional for all
  objects $x$ and zero on all but finitely many connected components.
\end{proposition}
\begin{proof} It remains to prove the ``if'' direction. It does no
  harm to replace $\CG$ by an equivalent groupoid which is the
  disjoint union of groups. On doing so, we note that $M$ is a finite
  direct sum of objects $N$ of $[\CC^{\mathrm{op}}, \CV]$ with $Nx=0$
  for all but one $x$. This reduces the problem to basic facts in the
  linear representation theory of the single finite group
  $G_x = \CG(x,x)$. Using Maschke's Theorem in the form that
  monomorphisms split in $[G_x^{\mathrm{op}}, \CV]$, we deduce that
  every finite-dimensional representation is a direct sum of
  irreducible representations. Using Maschke's Theorem in the form
  that epimorphisms split in $[G_x^{\mathrm{op}}, \CV]$, we deduce
  that every irreducible representation is a retract of the group
  algebra $k G_x$ with the regular action. However, $k G_x$ is the
  representable functor in $[G_x^{\mathrm{op}}, \CV]$ and so is
  Cauchy. This implies that our finite-dimensional representation is a
  finite direct sum of retracts of representables and therefore is
  Cauchy in $[G_x^{\mathrm{op}}, \CV]$ and hence in
  $[\CG^{\mathrm{op}}, \CV]\cong [\CC^{\mathrm{op}}, \CV]$. Yet, $M$
  is a finite direct sum of such objects and so Cauchy.
\end{proof}

{\small \vskip0.5\baselineskip \noindent\textbf{Declarations}. \emph{Data
availability}: no data was generated in the preparation of this
manuscript. \emph{Author contribution}: equal collaboration between the two
authors. \emph{Ethics}: not applicable. \emph{Funding}: both
authors acknowledge the support of Australian Research Council
Discovery Grants DP160101519 and DP190102432. \emph{Competing interests}: no competing interests.}

\appendix

\end{document}